\documentclass[12pt]{article}
\usepackage{amsmath}
\usepackage{amssymb}
\usepackage{amsmath,amsthm,amscd,amsfonts,amssymb,mathrsfs}
\usepackage[usenames]{color}

\usepackage{cite}
\usepackage{titlesec}

\allowdisplaybreaks

\numberwithin{equation}{section}
\titleformat*{\section}{\large\bfseries}

\newtheorem{thm}{Theorem}[section]
\newtheorem{lem}[thm]{Lemma}

\newtheorem{prop}[thm]{Proposition}

\newtheorem{rem}[thm]{Remark}
\newtheorem{hyp}[thm]{Hypothesis}

\begin{document}
\title{Smoothness of integrated density of states of the Anderson model on Bethe lattice in high disorder}
\author{Dhriti Ranjan Dolai\\
Indian Institute of Technology Dharwad, \\
Dharwad  - 580011, India.\\
email: dhriti@iitdh.ac.in\\~\\
M. Krishna\\
Ashoka University, Plot 2, Rajiv Gandhi Education City,\\
Rai, Haryana 131029, India\\
email: krishna.maddaly@ashoka.edu.in}
\maketitle
\noindent {\bf Abstract:} In this work we consider the Anderson model on Bethe lattice and prove that the integrated density of states (IDS) is as smooth as the single site distribution (SSD), in high disorder.
\section{Introduction}
The Bethe lattice $\mathbb{B}$ is an infinite connected graph with no closed loops and a fixed degree $K+1$ (number of nearest neighbours) at each vertex. The degree is called the coordination number and the connectivity $K$ is one less the coordination number. Distance between any two vertex $n_1$ and $n_2$ is denoted by $d(n_1,n_2)$ and it is the length of shortest path connecting $n_1$ and $n_2$. We denote $\ell^2(\mathbb{B})$ to be the Hilbert space 
$\bigg\{u:\mathbb{B}\to \mathbb{C}; \displaystyle \sum_{n\in\mathbb{B}}|u(n)|^2<\infty \bigg\}$. 
The Anderson Model on Bethe lattice is a random Schr\"{o}dinger operator $H^\omega$ on the Hilbert space $\ell^2(\mathbb{B})$ and it is given by 
\begin{align}
\label{model}
H^\omega &=\Delta+\lambda V^\omega,~\lambda>0,~\omega\in\Omega,\\
(\Delta u)(n) &=\sum_{k:d(k,n)=1}u(k),~ u=\{u(n)\}_{n\in\mathbb{B}}\in\ell^2(\mathbb{B}),\nonumber\\
(V^\omega u)(n) &=\omega_nu(n) \nonumber,
\end{align}
where $\{\omega_n\}_{n\in\mathbb{B}}$ are i.i.d real random variables with common distribution $\mu$ and it has a compact support. Consider the probability space $\big(\mathbb{R}^{\mathbb{B}}, \mathcal{B}_{\mathbb{R}^{\mathbb{B}}}, \mathbb{P} \big)$, where $\mathbb{P}=\underset{n\in\mathbb{B}}{\otimes}\mu$ is constructed via the Kolmogorov theorem. We refer to this probability space as $\big(\Omega, \mathcal{B}_\Omega, \mathbb{P}\big)$ and denote $\omega=(\omega_n)_{n\in\mathbb{B}}\in \Omega$.
The operator $\Delta$ is known as the graph Laplacian with diagonal term removed and the potential $V^\omega$ is the multiplication operator on $\ell^2(\mathbb{B})$ by the sequence $\{\omega_n\}_{n\in\mathbb{B}^d}$. We
note that the operators $\{H^\omega\}_{\omega\in\Omega}$ are bounded self-adjoint. It follows from the ergodicity (see \cite[appendix]{AK} and \cite{CL, Pa}) that the spectrum of the random operators $H^\omega$ is given by $\sigma(H^\omega)=[-2\sqrt{K}, 2\sqrt{K}]+ \lambda ~supp~\mu$, a.e $\omega$.\\~\\
First, we fix a vertex (we can choose any one) as the origin and denote it by $0$. Now we define the finite sub-graph $\Lambda_L$ and its boundary $\partial \Lambda_L$ as
$$\Lambda_L=\big\{n\in\mathbb{B}: d(n,0)\leq L \big\}~\&~\partial\Lambda_L=\big\{(n,m)\in\Lambda_L\times\big(\mathbb{B}\setminus \Lambda_L\big):d(n,m)=1 \big\}.$$
Let $\chi_{_L}$  and $P_n$ denote the orthogonal projection onto the subspace $\ell^2(\Lambda_L)$ and $\ell^2(\{\delta_n\})$, respectively. Here $\{\delta_n\}_{n\in\mathbb{B}}$ is the standard basis for the Hilbert space $\ell^2(\mathbb{B})$ . Define the finite matrix $H^\omega_L$ of size $|\Lambda_L|$ as
$$H^\omega_L=\Delta_L+V^\omega_L,~\Delta_L=\chi_{_L}\Delta\chi_{_L},~V^\omega_L=\displaystyle\sum_{n\in\Lambda_L}\omega_n P_n. $$
We consider the   integrated density of states (IDS) $\mathcal{N}(\cdot)$, as it is given in \cite[equation (1.4)]{AK}, namely
\begin{align}
\label{ids}
\mathcal{N}(E):=\mathbb{E}\big(\langle \delta_0, E_{H^\omega}(-\infty, E] \delta_0\big),
\end{align}
where $E_A(\cdot)$ denote the spectral measure of a self-adjoint operator $A$.
The probability measure $\nu(\cdot):=\mathbb{E}\big(\langle \delta_0, E_{H^\omega}(\cdot) \delta_0\big)$ is known as the density of states measure (DOSm). In the Hypothesis \ref{Hyp} (below), we assume the absolute continuity of $\mu$, therefore the Wegner estimate (see \cite{JFA}) will give the absolute continuity of the DOSm and in that case we have $d\nu(x)=\rho(x)dx$ and $\mathcal{N}'(x)=\rho(x)$.\\~\\
 In this paper we are interested in the degree of smoothness of $\rho$,  for which we make the following assumptions on $H^\omega$ (\ref{model}) :
\begin{hyp}
\label{Hyp}
~
\begin{enumerate}
 \item [(a)] The single site distribution (SSD) $\mu$ is absolutely continuous with respect to the Lebesgue measure on $\mathbb{R}$, i.e $d\mu(x)=g(x)dx$ and we also assume the function $g\in C^m_c(\mathbb{R})$, set of all continuously differentiable function up to the order $m$, with compact support.
 \item [(b)] Let the compact interval $J\subset \mathbb{R}$ be in the region of  exponential localization. Stated precisely, there exists a constant $C_s>0$ such that
 \begin{equation}
 \label{expo}
\sup_{\Re(z)\in J,~\Im(z)>0}\mathbb{E}\bigg(\bigg|\big\langle \delta_{n_1},\big(H^\omega_L-z\big)^{-1} \delta_{n_2}\big\rangle\bigg|^s\bigg)\leq C_se^{-\xi_s d(n_1,n_2)}, ~C_s>0,
 \end{equation}
 for some $\xi_s>0,~0<s<1$, for any $n_1,n_2\in\mathbb{B}$ and $\forall~L>0$. The decay rate $\xi_s$ satisfies the condition $\xi_s>2\ln(K+1)\big(3+p\big)$ for some $1\leq p\leq m$. We also assume that the constants $C_s,~\xi_s$ do not change if we replace the distribution $g$ with its derivatives at finitely many vertices.
\end{enumerate}
\end{hyp}
\begin{rem}
Exponential decay in (\ref{expo}) can be obtained with established techniques developed to prove the localization of Anderson model via fractional moments methods of the Green's function.  We refer to Aizenman \cite{A}, Aizenman-Molchanov \cite{AM} (see also Aizenman et al \cite{ASFH}) for more details. With the same method one can make $\xi_s$ (the decay rate in the exponential (\ref{expo})) arbitrary large by taking the disorder parameter $\lambda$ large enough, in the definition of $H^\omega$ (\ref{model}). Going through the proofs given in \cite{A, AM} and keep track of the constants, one can also realize that the assumption about changing the distribution at finitely many sites is also valid.
\end{rem}
\noindent Now we are ready to state our main result:
 \begin{thm}
 \label{main}
 Consider the operator $H^\omega$ given in (\ref{model}) satisfying the hypothesis \ref{Hyp}, then  the integrated density of states $\mathcal{N}(\cdot)$ is continuously differentiable up to the order $p$ in the interval $J$. In other words, 
 \begin{equation}
 \label{diff}
 \mathcal{N}(E)=\mathbb{E}\big(\langle \delta_0, E_{H^\omega}(-\infty, E] \delta_0\big)\in C^p(J),~
 \text{for some}~1\leq p\leq m.
 \end{equation}
 \end{thm}
 \noindent Several results are known about the smoothness of the integrated density of states (IDS) of the Anderson model on $\ell^2(\mathbb{Z}^d)$, under different types of assumptions on the single site distribution (SSD). We refer to Constantinescu et al. \cite{CFS}, Bovier et al \cite{BCKP}, Kaminaga et al \cite{KKS}, Bellissard-Hislop \cite{BH}, Carmona-Lacroix \cite[Corollary VI.3.2]{CL} for more details.
It is also possible to calculate the exact expression of the IDS for different models, when the SSD is Cauchy, we refer to Lloyd \cite{L}, Carmona-Lacroix \cite[problem VI.5.5]{CL} and Kirsch-Krishna \cite{KK} for more details.  \\
 For the one-dimensional random Schr\"{o}dinger operators, much more rich literature can be found in Companino-Klein \cite{CK}, March-Sznitman \cite{MS},  Simon-Taylor \cite{ST},  Speis \cite{S}, Klein-Speis \cite{KS, KS1}, Klein et al \cite{KLS} and Glaffig \cite{G}. Brodie, in his doctoral thesis \cite{Ben} investigated the smoothness of the density of states (DOS) for the random band matrix.\\
 The higher order differentiability of the integrated density of states (IDS) for continuum model on $\L^2(\mathbb{R}^d)$ was obtained by  Dolai et al \cite{DKM}. Although, their methods can also be applied to a large class of discrete model, but as $|\Lambda_L|$, the volume of the finite sub-graph (of Bethe lattice) increases exponentially as $L$ gets larger, the differentiability of IDS for Bethe lattice model is not immediate from \cite{DKM}.\\
 In \cite{AK}, Acosta-Klein showed the analyticity of the IDS for the Bethe lattice model when the single site distribution (SSD) is Cauchy or close to Cauchy (in some function space). The result in \cite{AK} is also valid in the region of absolutely continuous spectrum.\\
 In this work, we prove that for the Bethe lattice model, the integrated density of states (IDS) and the single site distribution (SSD) have the same degree of smoothness, in high disorder.
 \section{Proof of the results}
 In this section we will give the proof of our main result Theorem \ref{main}. Before going to the proof of this theorem, we need some pre requisite results. First, we show that the fractional moments of the difference of two resolvents of finite volume restriction of $H^\omega$ to the sub-graphs $\Lambda_{L+1}$ and $\Lambda_L$ is decay exponentially, as $L$ increases.\\
Let $G^\omega_L(z;n_1,n_2)$ denote the Green's function of the operator $H^\omega_L$ and it is defined by $G^\omega_L(z;n_1,n_2):=\big\langle \delta_{n_1},\big(H^\omega_L-z\big)^{-1} \delta_{n_2}\big\rangle$.
\begin{lem}
\label{resol-diff}
Under the condition [(b), Hypothesis \ref{Hyp}], the expected value of the difference of the resolvents can be bounded as, for some $B_s>0$ 
\begin{equation}
\label{resolvt-est}
\sup_{\Re(z)\in J,~\Im(z)>0}\mathbb{E}\bigg(\big|G^\omega_{L+1}(z;0,0)-G^\omega_L(z;0,0)\big|^{\frac{s}{2}}\bigg)\leq B_s e^{-(L+1)\big[\frac{\xi_s}{2} -\frac{2L+3}{L+1}\ln(K+1)\big]}.
\end{equation}
\end{lem}
\begin{proof}
Let's begin with the resolvent identity
\begin{align}
\label{resolv}
& (H^\omega_{L+1}-z)^{-1}- (H^\omega_L-z)^{-1}\nonumber\\
& \qquad=(H^\omega_{L+1}-z)^{-1}\bigg[(H^\omega_L-z)-(H^\omega_{L+1}-z) \bigg](H^\omega_L-z)^{-1}\nonumber\\
& \qquad= (H^\omega_{L+1}-z)^{-1}\bigg[\chi_{_L}\Delta\chi_{_L} -\chi_{_{L+1}}\Delta\chi_{_{L+1}}-\sum_{n\in \Lambda_{L+1}\setminus \Lambda_L}\omega_nP_n\bigg](H^\omega_L-z)^{-1}\nonumber\\
& \qquad = (H^\omega_{L+1}-z)^{-1}\bigg( \sum_{n\in \Lambda_{L+1}\setminus \Lambda_L}P_n\bigg)\Delta \chi_{_L}(H^\omega_L-z)^{-1}.
\end{align}
In the last line, we wrote $\chi_{_{L+1}}=\chi_{_L}+\sum_{n\in \Lambda_{L+1}\setminus \Lambda_L} P_n$ and also we have used the fact that the operator $\bigg(\displaystyle \sum_{n\in \Lambda_{L+1}\setminus \Lambda_L}P_n\bigg)(H^\omega_L-z)^{-1}\equiv 0$.\\
Now using the definition of the Laplacian $\Delta$ in (\ref{resolv}), we write
\begin{align}
\label{resolv-1}
& G^\omega_{L+1}(z;0,0)-G^\omega_L(z;0,0)\nonumber\\
& \qquad =-\sum_{m\in\Lambda_L}\sum_{\substack{k:d(k,m)=1\\ k\in\Lambda_{L+1}\setminus\Lambda_L} }  
\big\langle \delta_m, (H^\omega_L-z)^{-1}\delta_0 \big\rangle \big\langle \delta_0, (H^\omega_{L+1}-z)^{-1}\delta_k \big\rangle.
\end{align}
A bound for the number of vertex in $\Lambda_L$ can be given by $|\Lambda_L|\leq D (K+1)^{L+1}$, similarly 
$\big| \Lambda_{L+1}\setminus\Lambda_L\big|\leq D (K+1)^{L+2}$, $D$ is independent of $L$.\\
Now take the fractional power of the both side of the equation (\ref{resolv-1}),
\begin{align}
\label{resolv-2}
& \big|G^\omega_{L+1}(z;0,0)-G^\omega_L(z;0,0)\big|^{\frac{s}{2}}\nonumber\\
& \qquad \leq\sum_{m\in\Lambda_L}\sum_{\substack{k:d(k,m)=1\\ k\in\Lambda_{L+1}\setminus\Lambda_L} }  
\big|\big\langle \delta_m, (H^\omega_L-z)^{-1}\delta_0 \big\rangle\big|^{\frac{s}{2}} \big|\big\langle \delta_0, (H^\omega_{L+1}-z)^{-1}\delta_k \big\rangle\big|^{\frac{s}{2}}.
\end{align}
After taking the expectation both side of the above inequality (\ref{resolv-2}), we use Cauchy-Schwartz inequality to write
\begin{align}
\label{resolv-3}
& \mathbb{E}\bigg(\big|G^\omega_{L+1}(z;0,0)-G^\omega_L(z;0,0)\big|^{\frac{s}{2}}\bigg)\nonumber\\
&\qquad \leq\sum_{\substack{m\in\Lambda_L\\ k\in\Lambda_{L+1}\setminus\Lambda_L\\ k:d(k,m)=1 }}  \mathbb{E}\bigg(
\big|\big\langle \delta_m, (H^\omega_L-z)^{-1}\delta_0 \big\rangle\big|^{\frac{s}{2}} \big|\big\langle \delta_0, (H^\omega_{L+1}-z)^{-1}\delta_k \big\rangle\big|^{\frac{s}{2}}\bigg)\nonumber\\
& \qquad\leq\sum_{\substack{m\in\Lambda_L\\ k\in\Lambda_{L+1}\setminus\Lambda_L\\ k:d(k,m)=1 }} 
\mathbb{E}\bigg( \big|\big\langle \delta_m, (H^\omega_L-z)^{-1}\delta_0 \big\rangle\big|^s\bigg)^{\frac{1}{2}} 
\mathbb{E}\bigg( \big|\big\langle \delta_0, (H^\omega_{L+1}-z)^{-1}\delta_k \big\rangle\big|^s\bigg)^{\frac{1}{2}} \nonumber\\
& \qquad\leq D^2 C_s^2(K+1)^{2L+3}e^{-\frac{\xi_s}{2}(L+1)}.
\end{align}
In the last line we have used the volume estimation of $\Lambda_L,~\Lambda_{L+1}\setminus\Lambda_L$ and the fractional moments given in (\ref{expo}). Now (\ref{resolvt-est}) is immediate.
\end{proof}
\noindent We can always conclude the differentiability of the density of an absolutely continuous measure from the derivatives of its Borel transform.
\begin{lem}
\label{equi}
Let $d\tilde{\mu}(x)=f(x)dx$ be an absolutely continuous measure on the real line and $I\subset \mathbb{R}$ an interval. Denote $F(z):=\int_{\mathbb{R}}\frac{1}{x-z}d\tilde{\mu}(x)$ to be the Borel transformation of $\tilde{\mu}(\cdot)$. Then for any $m\in\mathbb{N}$
\begin{equation}
\label{diff-borel}
\underset{x\in I}{ess~sup}~\bigg|\frac{d^m}{dx^m}f(x)\bigg|<\infty~\text{whenever}~\sup_{\Re(z)\in I,~\Im(z)>0}\bigg|\Im{\bigg(\frac{d^m}{dz^m}F(z) \bigg)}\bigg|<\infty.
\end{equation}
\end{lem} 
\noindent The proof can be found in \cite[Lemma A.1]{DKM}.\\~\\
The above Lemma will give the higher order differentiability of $\rho$, the density of states function (DOS), once we show that the imaginary part of the higher order derivatives of the expected value of the Green's function of $H^\omega$ is bounded on the upper half plane.
\begin{prop}
\label{difer-finite-box}
Under the condition [(a) Hypothesis \ref{Hyp}], the imaginary part of the higher order derivative of the Green's function
$G^\omega_{M}(z;0,0)$ is bounded on the upper half of the complex plane
\begin{equation}
\label{bound-finite-box}
\max_{1\leq \ell\leq p}\sup_{\Im(z)>0}\bigg|\Im{\bigg(\frac{d^\ell}{dz^\ell} \mathbb{E}\big(G^\omega_{M}(z;0,0) \big)\bigg)}\bigg|<\infty~~\text{for}~~1\leq p\leq m,
\end{equation}
here $M$ is a fixed positive integer.
\end{prop}
\begin{proof}
Let $z=E+i\epsilon,~\epsilon>0$, we can re-write the expression for $H^\omega_M-z$ as
\begin{align}
\label{trans}
H^\omega_{M}-z &=\Delta_{M}+\sum_{n\in\Lambda_{M}}(\omega_n-E) P_n-i\epsilon\\
                          &=\Delta_{M}+\sum_{n\in\Lambda_{M}}\tilde{\omega}_n P_n-i\epsilon,~~\tilde{\omega}_n=\omega_n-E\nonumber\\
                          &=H^{\tilde{\omega}}_{M}-i\epsilon,~~\tilde{\omega}=(\tilde{\omega}_n)_{n\in\Lambda_M}.
\end{align}
We write the expected value of the Green's function as the integration
\begin{align}
\label{repe}
& \mathbb{E}\big(G^\omega_M(z;0,0) \big)\nonumber\\
& \qquad=\int_{_{\mathbb{R}^{|\Lambda_{_M}|}}}\langle \delta_0, \big( H^\omega_M-z\big)^{-1}\delta_0\big \rangle
\prod_{n\in\Lambda_M}g(\omega_n) d\omega_n\nonumber\\
& \qquad=\int_{_{\mathbb{R}^{|\Lambda_{_M}|}}}\langle \delta_0, \big( H^{\tilde{\omega}}_M-i\epsilon\big)^{-1}\delta_0\big \rangle\prod_{n\in\Lambda_M}g(\tilde{\omega}_n+E) d\tilde{\omega}_n.
\end{align}
Using analyticity of $G^\omega_M(z;0,0)$ and the representation (\ref{repe}), we can write its higher order derivatives in terms of the product of the derivatives of $g$,
\begin{align}
\label{change-var}
 & \frac{d^\ell}{dz^\ell} \mathbb{E}\big(G^\omega_M(z;0,0) \big),~~0\leq\ell\leq m\nonumber\\
&=\frac{d^\ell}{dE^\ell} \mathbb{E}\big(G^\omega_M(z;0,0) \big),~~z=E+i\epsilon,~\epsilon>0\nonumber\\
&=\sum_{\substack{k_0+\cdots+k_{|\Lambda_M|-1}=\ell\\\forall j,~ k_j\ge 0} }
\binom{\ell}{k_0,\cdots,k_{|\Lambda_M|-1}}\int_{_{\mathbb{R}^{|\Lambda_{_M}|}}} \big\langle \delta_0, \big( H^{\tilde{\omega}}_M-
i\epsilon\big)^{-1}\delta_0\big \rangle
\nonumber\\
&\qquad\qquad\qquad\qquad\qquad \times\bigg(\prod_{n=0}^{|\Lambda_M|-1}g^{(k_n)}(\tilde{\omega}_n+E) d\tilde{\omega}_n\bigg).
\end{align}
Use resolvent identity for one dimensional perturbation to write
\begin{align*}
\big\langle \delta_0, \big( H^{\tilde{\omega}}_M-i\epsilon\big)^{-1}\delta_0\big \rangle&=G^{\tilde{\omega}}_M(i\epsilon;0,0)\nonumber\\
&=\frac{1}{\lambda\tilde{\omega}_0+\big(G^{\tilde{\omega}\setminus\tilde{\omega}_0}_M(i\epsilon;0,0)\big)^{-1}},~\nonumber\\
&=\frac{1}{\lambda\tilde{\omega}_0+A-iB}\nonumber.
\end{align*}
here $G^{\tilde{\omega}\setminus\tilde{\omega}_0}_M(i\epsilon;0,0)=(A-iB)^{-1}$ denote the Green's function for the matrix $H^{\tilde{\omega}\setminus\tilde{\omega}_0}_M:=H^{\tilde{\omega}}_M-\lambda\omega_0P_0$. It is easy check that $B>0$ for $\Im(z)>0$.\\
Since $g\in C^m_c(\mathbb{R})$, therefore we estimate the integral
\begin{align}
\label{bdd-1}
&\bigg|\Im\bigg(\int_\mathbb{R} \big\langle \delta_0, \big( H^{\tilde{\omega}}_M-i\epsilon\big)^{-1}\delta_0\big \rangle g^{(k_0)}(\tilde{\omega}_0+E)d\omega_0\bigg)\bigg|\nonumber\\
&\qquad=\int_\mathbb{R}\frac{B}{(\lambda\omega_0+A)^2+B^2}\big| g^{(k_0)}(\tilde{\omega}_0+E)\big|d\omega_0,\nonumber\\
&\qquad \leq\big\|g^{(k_0)} \big\|_\infty \frac{\pi}{\lambda^2},~~~0\leq k_0\leq m.
\end{align}
The integral in (\ref{change-var}) can be written as
\begin{align}
\label{change-order}
& \int_{_{\mathbb{R}^{|\Lambda_{_M}|}}} \big\langle \delta_0, \big( H^{\tilde{\omega}}_M-i\epsilon\big)^{-1}\delta_0\big \rangle
\bigg(\prod_{n=0}^{|\Lambda_M|-1}g^{(k_n)}(\tilde{\omega}_n+E) d\tilde{\omega}_n\bigg)\nonumber\\
&\qquad\qquad=\int_{_{\mathbb{R}^{|\Lambda_{_M}|-1}}}
\bigg(\int_{\mathbb{R}}\big\langle \delta_0, \big( H^{\tilde{\omega}}_M-i\epsilon\big)^{-1}\delta_0\big \rangle g^{(k_0)}(\omega_0+E)d\omega_0\bigg)\nonumber\\
&\qquad \qquad \qquad \qquad\qquad\qquad\times\bigg(\prod_{n\neq 0}g^{(k_n)}(\tilde{\omega}_n+E) d\tilde{\omega}_n\bigg).
\end{align}
Now (\ref{bound-finite-box}) can be easily obtained, once we use (\ref{change-order}) and (\ref{bdd-1}) in (\ref{change-var}).
\end{proof}
\noindent Using \cite[Theorem 2.2]{DKM} we can estimate the average of the  difference of the resolvents by its fractional moments.
\begin{prop}
\label{bdd-frac}
Under the assumption [(a) Hypothesis \ref{Hyp}], we have
\begin{align}
\label{cmp-result}
& \bigg|\int \langle \delta_0, \big(H^\omega_{L+1}-z  \big)^{-1}- \big(H^\omega_{L}-z  \big)^{-1}\delta_0\rangle g(\omega_0)d\omega_0\bigg|,~~~\Im(z)>0\nonumber\\
&\qquad \leq \Xi \int \bigg| \langle \delta_0, \big(H^\omega_{L+1}-z  \big)^{-1}- \big(H^\omega_{L}-z  \big)^{-1}\delta_0\rangle\bigg|^{\frac{s}{2}}\phi_{R}(\omega_0)d\omega_0.
\end{align}
Here $0<s<2$ and $\phi_R$ is a translation of a smooth indicator function $\chi_R$ on a bounded interval which contain the support of the density $g$. The constant $\Xi>0$ is independent of $\omega$, $z$ and $L$ but depends on $s$ and $g$.
\end{prop}
\begin{proof}
Using the notations of \cite[Theorem 2.2]{DKM} we define $F:=|\delta_0\rangle\langle \delta_0|$, $\rho_1:=g$, $A:=H_{L+1}^{\omega\setminus\omega_0}=H_{L+1}^\omega-\omega_0~|\delta_0\rangle\langle \delta_0|$ and $B:=H_{L}^{\omega\setminus\omega_0}=H_L^\omega-\omega_0~|\delta_0\rangle\langle \delta_0|$. 
Now (\ref{cmp-result}) is immediate from \cite[(2) of Theorem 2.2]{DKM} .
\end{proof}
\begin{rem}
\label{cmp-lemma}
For more details on the above result and the properties of $\phi_R$, $\chi_R$ and $\Xi$, we refer to \cite[Theorem 2.2]{DKM}.
\end{rem}
\noindent Now the above proposition will give the exponential decay of the higher order derivatives of the expected value of the difference of the resolvents $G^\omega_{L+1}(z;0,0)$ and $G^\omega_L(z;0,0)$.
\begin{prop}
Under the Hypothesis \ref{Hyp}, we have
\begin{align}
\label{exp-diff}
& \sup_{\Re(z)\in J,~\Im(z)>0}\bigg|\frac{d^\ell}{dz^\ell}\mathbb{E}\bigg(G^\omega_{L+1}(z;0,0)-G^\omega_L(z;0,0)\bigg)\bigg|,~~1\leq\ell\leq m\nonumber\\
&\qquad\qquad\qquad\leq\tilde{C}_se^{-(L+1)\big[\frac{\xi_s}{2} -\frac{\ln(K+1)}{L+1}\big( (2L+3)+(L+2)(\ell+1)\big)\big]},~\tilde{C}_s>0.
\end{align}
\end{prop}
\begin{proof}
As in the proof of the inequality (\ref{change-var}), we write
\begin{align}
\label{change-var-1}
 & \frac{d^\ell}{dz^\ell} \mathbb{E}\big(G^\omega_{L+1}(z;0,0) \big),~~0\leq\ell\leq m\nonumber\\
&=\frac{d^\ell}{dE^\ell} \mathbb{E}\big(G^\omega_{L+1}(z;0,0) \big),~~z=E+i\epsilon,~\epsilon>0\nonumber\\
&=\sum_{\substack{k_0+\cdots+k_{|\Lambda_{L+1}|-1}=\ell\\\forall j,~ k_j\ge 0} }
\binom{\ell}{k_0,\cdots,k_{|\Lambda_{L+1}|-1}}\int_{_{\mathbb{R}^{|\Lambda_{_{L+1}}|}}} \big\langle \delta_0, \big( H^{\tilde{\omega}}_{L+1}-i\epsilon\big)^{-1}\delta_0\big \rangle
\nonumber\\
&\qquad\qquad\qquad\qquad\qquad \times\bigg(\prod_{n=0}^{|\Lambda_{L+1}|-1}g^{(k_n)}(\tilde{\omega}_n+E) d\tilde{\omega}_n\bigg).
\end{align}
We also write
\begin{align}
\label{change-var-2}
 & \frac{d^\ell}{dz^\ell} \mathbb{E}\big(G^\omega_{L}(z;0,0) \big),~~0\leq\ell\leq m\nonumber\\
&=\frac{d^\ell}{dE^\ell} \mathbb{E}\big(G^\omega_{L}(z;0,0) \big),~~z=E+i\epsilon,~\epsilon>0\nonumber\\
&=\sum_{\substack{k_0+\cdots+k_{|\Lambda_{L}|-1}=\ell\\\forall j,~ k_j\ge 0} }
\binom{\ell}{k_0,\cdots,k_{|\Lambda_{L}|-1}}\int_{_{\mathbb{R}^{|\Lambda_{_{L}}|}}} \big\langle \delta_0, \big( H^{\tilde{\omega}}_L-i\epsilon\big)^{-1}\delta_0\big \rangle
\nonumber\\
&\qquad\qquad\qquad\qquad\qquad \times\bigg(\prod_{n=0}^{|\Lambda_{L}|-1}g^{(k_n)}(\tilde{\omega}_n+E) d\tilde{\omega}_n\bigg).
\end{align}
The above integrand is independent of the random variables $\{\omega_n\}_{n\in\Lambda_{L+1}\setminus\Lambda_L}$ and under the assumption [(a), Hypothesis] $\int g^{(j)}(x)dx=\delta_{0,j},~1\leq j \leq m$. We use this two facts to write the (\ref{change-var-2}) as
\begin{align}
\label{change-var-3}
 & \frac{d^\ell}{dz^\ell} \mathbb{E}\big(G^\omega_{L}(z;0,0) \big),~~0\leq\ell\leq m\nonumber\\
&=\sum_{\substack{k_0+\cdots+k_{|\Lambda_{L+1}|-1}=\ell\\\forall j,~ k_j\ge 0} }
\binom{\ell}{k_0,\cdots,k_{|\Lambda_{L+1}|-1}}\int_{_{\mathbb{R}^{|\Lambda_{_{L+1}}|}}} \big\langle \delta_0, \big( H^{\tilde{\omega}}_L-i\epsilon\big)^{-1}\delta_0\big \rangle
\nonumber\\
&\qquad\qquad\qquad\qquad\qquad \times\bigg(\prod_{n=0}^{|\Lambda_{L+1}|-1}g^{(k_n)}(\tilde{\omega}_n+E) d\tilde{\omega}_n\bigg).
\end{align}
Now the derivatives of the difference can be written as
\begin{align*}
& \frac{d^\ell}{dz^\ell} \mathbb{E}\bigg(G^\omega_{L+1}(z;0,0)-G^\omega_{L}(z;0,0) \bigg),~~0\leq\ell\leq m\nonumber\\
&=\sum_{\substack{k_0+\cdots+k_{|\Lambda_{L+1}|-1}=\ell\\\forall j,~ k_j\ge 0} }
\binom{\ell}{k_0,\cdots,k_{|\Lambda_{L+1}|-1}}\nonumber\\
&\qquad\qquad\qquad\qquad\times\int_{_{\mathbb{R}^{|\Lambda_{_{L+1}}|}}}\bigg( \big\langle \delta_0, \big( H^{\tilde{\omega}}_{L+1}-i\epsilon\big)^{-1}-\big( H^{\tilde{\omega}}_L-i\epsilon\big)^{-1}\delta_0\big \rangle\bigg)\nonumber\\
&\qquad\qquad\qquad\qquad\qquad\qquad\qquad \times\bigg(\prod_{n=0}^{|\Lambda_{L+1}|-1}g^{(k_n)}(\tilde{\omega}_n+E) d\tilde{\omega}_n\bigg).\nonumber
\end{align*}
For $\ell<<L$, in the product $\displaystyle\prod_{n=0}^{|\Lambda_{L+1}|-1}g^{(k_n)}(\tilde{\omega}_n+E)$ at most $\ell$ number of places $g^{(k_n)}$ is different from $g$ and all the other remaining places we have $g^{(k_n)}=g$. Therefore, w.l.o.g we can always assume $g^{(k_0)}=g$. Now the above expression can be rewritten as
\begin{align}
\label{change-var-4}
&\frac{d^\ell}{dz^\ell} \mathbb{E}\bigg(G^\omega_{L+1}(z;0,0)-G^\omega_{L}(z;0,0) \bigg),~~0\leq\ell\leq m\nonumber\\
&\leq\sum_{\substack{k_0+\cdots+k_{|\Lambda_{L+1}|-1}=\ell\\\forall j,~ k_j\ge 0} }
\binom{\ell}{k_0,\cdots,k_{|\Lambda_{L+1}|-1}}\nonumber\\
&\qquad\times\int_{_{\mathbb{R}^{|\Lambda_{_{L+1}}|-1}}}\bigg(\int_{\mathbb{R}} \big\langle \delta_0, \big( H^{\tilde{\omega}}_{L+1}-i\epsilon\big)^{-1}-\big( H^{\tilde{\omega}}_L-i\epsilon\big)^{-1}\delta_0\big \rangle g^{(k_0)}(\tilde{\omega}_0+E)d\omega_0\bigg)
\nonumber\\
&\qquad\qquad\qquad\qquad\qquad\qquad\qquad \times\bigg(\prod_{n\neq 0}g^{(k_n)}(\tilde{\omega}_n+E) d\tilde{\omega}_n\big).
\end{align}
To estimate the above difference, we are going to use the Proposition \ref{bdd-frac}. As the function $g^{(k_0)}=g$ is the density of $\mu$, the single site distribution (SSD) we have,
\begin{align}
\label{change-var-5}
& \bigg|  \frac{d^\ell}{dz^\ell} \mathbb{E}\bigg(G^\omega_{L+1}(z;0,0)-G^\omega_{L}(z;0,0) \bigg)\bigg|,~~0\leq\ell\leq m\nonumber\\
&\leq\sum_{\substack{k_0+\cdots+k_{|\Lambda_{L+1}|-1}=\ell\\\forall j,~ k_j\ge 0} }
\binom{\ell}{k_0,\cdots,k_{|\Lambda_{L+1}|-1}}\times\nonumber\\
&\int_{_{\mathbb{R}^{|\Lambda_{_{L+1}}|-1}}}\bigg|\int_{\mathbb{R}} \big\langle \delta_0, \big( H^{\tilde{\omega}}_{L+1}-i\epsilon\big)^{-1}-\big( H^{\tilde{\omega}}_L-i\epsilon\big)^{-1}\delta_0\big \rangle g^{(k_0)}(\tilde{\omega}_0+E)d\omega_0\bigg|
\nonumber\\
&\qquad\qquad\qquad\qquad\qquad\qquad\qquad \times\bigg(\prod_{n\neq 0}\big|g^{(k_n)}(\tilde{\omega}_n+E)\big| d\tilde{\omega}_n\big)\nonumber\\
& \leq \Xi\sum_{\substack{k_0+\cdots+k_{|\Lambda_{L+1}|-1}=\ell\\\forall j,~ k_j\ge 0} }
\binom{\ell}{k_0,\cdots,k_{|\Lambda_{L+1}|-1}}\times\nonumber\\
&\int_{_{\mathbb{R}^{|\Lambda_{_{L+1}}|-1}}}\bigg(\int_{\mathbb{R}}\bigg| \big\langle \delta_0, \big( H^{\tilde{\omega}}_{L+1}-i\epsilon\big)^{-1}-\big( H^{\tilde{\omega}}_L-i\epsilon\big)^{-1}\delta_0\big \rangle\bigg|^{\frac{s}{2}}\phi_R(\tilde{\omega}_0+E)d\omega_0\bigg)
\nonumber\\
&\qquad\qquad\qquad\qquad\qquad\qquad\qquad \times\bigg(\prod_{n\neq 0}\big|g^{(k_n)}(\tilde{\omega}_n+E)\big| d\tilde{\omega}_n\big).\nonumber\\
&\leq \Xi~ C^\ell\sum_{\substack{k_0+\cdots+k_{|\Lambda_{L+1}|-1}=\ell\\\forall j,~ k_j\ge 0} }
\binom{\ell}{k_0,\cdots,k_{|\Lambda_{L+1}|-1}}\times\nonumber\\
&\int_{_{\mathbb{R}^{|\Lambda_{_{L+1}}|-1}}}\bigg(\int_{\mathbb{R}}\bigg| \big\langle \delta_0, \big( H^{\tilde{\omega}}_{L+1}-i\epsilon\big)^{-1}-\big( H^{\tilde{\omega}}_L-i\epsilon\big)^{-1}\delta_0\big \rangle\bigg|^{\frac{s}{2}}\tilde{g}_0(\tilde{\omega}_0+E)d\omega_0\bigg)
\nonumber\\
&\qquad\qquad\qquad\qquad\qquad\qquad\qquad \times\bigg(\prod_{n\neq 0}\tilde{g}_n(\tilde{\omega}_n+E) d\tilde{\omega}_n\big).
\end{align}
In the above (\ref{change-var-5}), we have used the notation $\tilde{g}_n(x)=\frac{|g^{(k_n)}(x)|}{\|g^{(k_n)} \|_1},~n\neq 0$ and $\tilde{g}_0(x)=\frac{\phi_R(x)}{\|\phi \|_1}$. As discussed earlier that in the product $\displaystyle\prod_{n\neq 0}\tilde{g}_n$, at most $\ell$ number of places $\tilde{g}_n\neq g$ and all the other remaining places $\tilde{g}_n$ will be $g$, with $\| g\|_1=1$.
Since $\phi_R$ is a compactly supported smooth function and $g$, the density of $\mu$ (SSD) is in $C^m_c(\mathbb{R})$, then there exist a constant $C>0$ such that $\displaystyle\max_{n:~0\leq k_n\leq m}\big\{\|g^{(k_n)} \|_1,~\|\phi_R \|_1\big\}\leq C$ holds.\\
Assume $\tilde{\mathbb{E}}(\cdot)$ denote the expectation with respect to the product probability measure $\displaystyle\prod_{n}\tilde{g}_n(\omega_n)d\omega_n$. Now the (\ref{change-var-5}) can be rewritten as
\begin{align}
\label{repe-1}
& \bigg| \frac{d^\ell}{dz^\ell} \mathbb{E}\bigg(G^\omega_{L+1}(z;0,0)-G^\omega_{L}(z;0,0) \bigg)\bigg|,~~0\leq\ell\leq m\nonumber\\
&\leq \Xi~ C^\ell\sum_{\substack{k_0+\cdots+k_{|\Lambda_{L+1}|-1}=\ell\\\forall j,~ k_j\ge 0} }
\binom{\ell}{k_0,\cdots,k_{|\Lambda_{L+1}|-1}}\nonumber\\
&\qquad\qquad\qquad\qquad\times\tilde{\mathbb{E}}\bigg( \bigg| \big\langle \delta_0, \big( H^\omega_{L+1}-z\big)^{-1}-\big( H^\omega_L-z\big)^{-1}\delta_0\big \rangle\bigg|^{\frac{s}{2}}\bigg),
\end{align}
here we have used the notation $z=E+i\epsilon$ and $\omega=\big(\tilde{\omega}_n+E\big)_{n\in\Lambda_{L+1}}$.\\
In view of the assumption [(b), Hypothesis \ref{Hyp}] and using the Lemma \ref{resol-diff}, we estimate the r.h.s of (\ref{repe-1})
\begin{align}
\label{repe-2}
& \bigg| \frac{d^\ell}{dz^\ell} \mathbb{E}\bigg(G^\omega_{L+1}(z;0,0)-G^\omega_{L}(z;0,0) \bigg)\bigg|\nonumber\\
&\qquad\qquad\qquad\leq \Xi ~C^\ell\tilde{C}\big| \Lambda_{L+1}\big|^{\ell+1}B_se^{-(L+1)\big[\frac{\xi_s}{2} -\frac{2L+3}{L+1}\ln(K+1)\big]}\nonumber\\
&\qquad\qquad\qquad\leq \Xi ~C^\ell\tilde{C} B_se^{-(L+1)\big[\frac{\xi_s}{2} -\frac{\ln(K+1)}{L+1}\big( (2L+3)+(L+2)(\ell+1)\big)\big]},
\end{align}
here we estimate the number of terms in the sum (\ref{repe-1}) by $\tilde{C}|\Lambda_{L+1}|^{\ell+1}$ and $|\Lambda_{L+1}|$, the number of vertex in $\Lambda_{L+1}$ can be bounded by a constant multiplication of $(K+1)^{L+2}$.
Hence the proposition.
\end{proof}
\noindent We also require the convergence of $G^\omega_L(z;0,0)$, the Green's function of finite volume approximation $H^\omega_L$ to $G^\omega(z;0,0)$, the Green's function of the full operator $H^\omega$. Although, for $\mathbb{Z}^d$ model, this convergence of Green's function is relatively straightforward, for the Bethe lattice model, it demands proof, and it was given in \cite[Proposition 1.2]{AK}. Using the same result \cite[Proposition 1.2]{AK}, we can also conclude the convergence of the derivatives of $G^\omega_L(z;0,0)$ to the derivatives of $G^\omega(z;0,0)$.
\begin{prop} For $\Im (z)>0$ and $\ell\in\mathbb{N}$, we have
\begin{equation}
\label{resol-convrg}
\lim_{L\to\infty}\mathbb{E}\big(G^\omega_L(z;0,0)  \big)=\mathbb{E}\big(G^\omega(z;0,0)  \big)
\end{equation}
and
\begin{equation}
\label{resol-der}
\lim_{L\to\infty}\frac{d^\ell}{dz^\ell}\mathbb{E}\big(G^\omega_L(z;0,0)  \big)=\frac{d^\ell}{dz^\ell}\mathbb{E}\big(G^\omega(z;0,0)  \big).
\end{equation}
\end{prop}
\begin{proof}
For the proof of (\ref{resol-convrg}), we refer to \cite[Proposition 1.2]{AK}.\\
Define $g_L(z):=\mathbb{E}\big(G^\omega_L(z;0,0)  \big)$ and it is analytic function on the upper half of the complex plane. Now we use the Cauchy integral formula to write $g^{(\ell)}_L$, the $\ell$th order derivative of $g_L$ as
\begin{equation}
\label{cauchy}
g_L^{(\ell)}(z)=\frac{\ell!}{2\pi i}\int_{\gamma}\frac{g_L(w)}{(w-z)^{\ell+1}}dw, 
\end{equation}
here $\gamma$ be a circle with center at $z$, oriented counter clockwise and it is lie inside the upper half of the complex plane.
Now the convergence (\ref{resol-der}) is immediate once we use (\ref{resol-convrg}) and dominated convergence theorem in (\ref{cauchy}).
\end{proof}
\noindent In view of the Lemma \ref{equi}, to prove the higher order differentiability of $\rho$, the density of states function (DOSf), we only need to show the imaginary part of the higher order derivatives of $\mathbb{E}\big(G^\omega(z;0,0)\big)$, expected value of the Green's function of the full operator $H^\omega$ is uniformly bounded in the upper half of the complex plane.\\~\\
\noindent {\bf Proof of the Theorem \ref{main}:} Since $G^\omega_{L+1}(z;0,0)$ converges to $G^\omega(z;0,0)$ (see (\ref{resol-convrg})), therefore for a fix $M\in\mathbb{N}$ (large enough), using telescoping series we write
\begin{align}
\label{appro}
&\mathbb{E}\big(G^\omega(z;0,0) \big)\nonumber\\
&=\mathbb{E}\big(\big\langle \delta_0, (H^\omega-z)^{-1}\delta_0\big\rangle \big)\nonumber\\
&=\lim_{L\to\infty}\mathbb{E}\big(G^\omega_{L+1}(z;0,0) \big)\nonumber\\
&=\sum_{L=M}^\infty\mathbb{E}\bigg(G^\omega_{L+1}(z;0,0)-G^\omega_{L}(z;0,0) \bigg)+G^\omega_M(z;0,0).
\end{align}
Using the convergence of the derivatives of $G_L^\omega(z;0,0)$, as in (\ref{resol-der}), we can write its derivative as the infinite series of the derivatives of its finite volume approximation
\begin{align}
\label{appro}
&\frac{d^\ell}{dz^\ell}\mathbb{E}\big(G^\omega(z;0,0) \big)\nonumber\\
&=\frac{d^\ell}{dz^\ell}\mathbb{E}\big(\big\langle \delta_0, (H^\omega-z)^{-1}\delta_0\big\rangle \big),~1\leq\ell\leq p\nonumber\\
&=\lim_{L\to\infty}\frac{d^\ell}{dz^\ell}\mathbb{E}\big(G^\omega_{L+1}(z;0,0) \big)\nonumber\\
&=\sum_{L=M}^\infty\frac{d^\ell}{dz^\ell}\mathbb{E}\bigg(G^\omega_{L+1}(z;0,0)-G^\omega_{L}(z;0,0) \bigg)+\frac{d^\ell}{dz^\ell}G^\omega_M(z;0,0).
\end{align}
 Since $\xi_s>2\ln(K+1)\big(3+p\big)$, assumed in [(b), Hypothesis \ref{Hyp}], then use of (\ref{exp-diff}) and (\ref{bound-finite-box}) in (\ref{appro}) will give
 \begin{align}
 \label{finiteness}
 \sup_{\Re(z)\in J,~\Im(z)>0}\bigg| \Im\bigg(\frac{d^\ell}{dz^\ell}\mathbb{E}\big(G^\omega(z;0,0) \big)\bigg)\bigg|<\infty~~\text{for}~~1\leq \ell\leq p.
 \end{align}
 The Green's function $G^\omega(z;0,0)$ is nothing but the Borel transformation of $\nu$, the density of states measure (DOSm), which is  absolutely continuous in our case and $\mathcal{N}$, the IDS is its distribution function. Hence an application of the Lemma \ref{equi} to the estimate (\ref{finiteness}) will give the result.
 \qed\\~\\
{\bf Acknowledgment}\\~\\
Dhriti Ranjan Dolai is partially supported by the INSPIRE Faculty Fellowship grant of DST, Govt of India.

\end{document}